\documentclass[hidelinks]{siamart0216}
\usepackage[T1]{fontenc}
\usepackage{mystyle}
\usepackage{todonotes}
\usepackage{algorithmicx}
\usepackage{algpseudocode}
\usepackage{subcaption}

\newcommand{\maxx}{x^+}

\newcommand{\conv}[1]{\mathrm{conv}\ens{#1}}
\DeclareMathOperator{\aff}{aff}
\newcommand{\X}{\mathbf{X}}
\newcommand{\MAX}{\mathbf{S}}
\newcommand{\Xl}{\X_{\lambda}}
\newcommand{\Xo}{\X_{0}}
\newcommand{\Sl}{\MAX_{\lambda}}

\DeclareMathOperator{\dom}{dom}
\DeclareMathOperator{\card}{card}

\newcommand{\TheTitle}{Maximal Solutions of Sparse Analysis Regularization}
\newcommand{\TheAuthors}{A. Barbara, A. Jourani and S. Vaiter}

\ifpdf
\hypersetup{
  pdftitle={\TheTitle},
  pdfauthor={\TheAuthors}
}
\fi

\headers{\TheTitle}{\TheAuthors}

\title{{\TheTitle}}

\author{A. Barbara%
  \thanks{Universit\'e de Bourgogne Franche-Comt\'e, Institut de Math\'ematiques de Bourgogne, UMR 5584 CNRS, (\email{\{barbara,jourani\}@u-bourgogne.fr})}%
  \and
  A. Jourani%
  \footnotemark[1]%
  \and
  S. Vaiter%
  \thanks{CNRS \& Universit\'e de Bourgogne Franche-Comt\'e, Institut de Math\'ematiques de Bourgogne, UMR 5584 CNRS, (\email{vaiter@u-bourgogne.fr})}
}

\begin{document}

\maketitle

\begin{abstract}
  This paper deals with the non-uniqueness of the solutions of an analysis-Lasso regularization.
  Most of previous works in this area is concerned with the case where the solution set is a singleton, or to derive guarantees to enforce uniqueness.
  Our main contribution consists in providing a geometrical interpretation of a solution with a maximal $D$-support, namely the fact that such a solution lives in the relative interior of the solution set.
  With this result in hand, we also provide a way to exhibit a maximal solution using a primal-dual interior point algorithm.
\end{abstract}

\begin{keywords}
  Lasso, analysis sparsity,  inverse problem, support identification, barrier penalization
\end{keywords}

\begin{AMS}
  90C25, 
  49J52  
\end{AMS}

\section{Introduction}

We consider the problem of estimating an unknown vector $x_0 \in \RR^n$ from noisy observations
\begin{equation}\label{eq:ip}
  y = \Phi x_0 + w \in \RR^q ,
\end{equation}
where $\Phi$ is a linear operator from $\RR^n$ to $\RR^q$ and $w$ is the realization of a noise.
This linear model is widely used in imaging for degradation such that entry-wise masking, convolution, etc, or in statistics under the name of linear regression.
Typically, the inverse problem associated to~\eqref{eq:ip} is ill-posed, and one should add additional information in order to recover at least an approximation of $x_0$.

During the last decade, sparse regularization in orthogonal basis has become a classical tool in the analysis of such inverse problem, in particular in imaging~\cite{chen1999atomi,mallat2009a-wav} or in statistics and machine learning~\cite{tibshirani1996regre}.
The sparsity of some coefficients $x \in \RR^n$ is measured using the counting function, or abusively $\ell^0$ norm, which reads
\begin{equation*}
  \norm{x}_0 = \Card (\supp(x)) \qwhereq \supp(x) = \enscond{i \in \ens{1,\dots,n}}{x_i \neq 0} ,
\end{equation*}
where $\supp(x)$ is coined the support of the vector $x$.
The associated regularization
\begin{equation*}
  \uArgmin{x \in \RR^n} \frac{1}{2} \norm{y - \Phi x}_2^2 + \lambda \norm{x}_0
\end{equation*}
is however known to be NP-hard~\cite{natarajan1995sparse}.
A first way to alievate this issue is to consider greedy methods, such as the Matching Pursuit~\cite{mallat1993matching} or derivation from it as the OMP~\cite{pati1993orthogonal}, CoSAMP~\cite{needel2009cosamp}, etc.
This will not be the concern of this paper which focus on one of its most popular convex relaxation through the $\lun$-norm.
More precisely, we consider the Lasso optimization problem~\cite{tibshirani1996regre} which reads
\begin{equation}
  \label{eq:lasso}
  \uArgmin{x \in \RR^n} \frac{1}{2} \norm{y - \Phi x}_2^2 + \lambda \normu{x} ,
\end{equation}
where the $\lun$-norm is defined as $\normu{x} = \sum_{i=1}^n \abs{x_i}$.

In this work, we consider a more general framework, known as the sparse analysis prior, cosparse prior or generalized Lasso.
The idea is to not measure the sparsity of the coefficients in an orthogonal basis, but in any \emph{dictionary}.
Formally, a dictionary $D$ is a linear operator from $\RR^p$ to $\RR^n$ which is defined through $p$ $n$-dimensional \emph{atoms} $d_i$ which may be redundant.
Using this dictionary, one can build an analysis regularization which reads $\norm{D^* \cdot}_1$ associated to the variational framework defined as
\begin{equation}
  \label{eq:p}
  \Xl = \uArgmin{x \in \RR^n}h(x)= \frac{1}{2} \norm{y - \Phi x}_2^2 + \lambda \normu{D^* x} .
\end{equation}
This framework is known in the signal processing community as sparse analysis regularization~\cite{elad2007analysis,vaiter2011robust} or cosparse regularization~\cite{nam2012cosparse}.
Probably the most popular example of analysis sparsity-inducing regularizer is the Total Variation which was introduced in~\cite{rudin1992nonlinear} in a continuous setting for denoising.
In the discrete setting, it corresponds to take $D^*$ as a discretization of a derivative operator.
In the context of one-dimensional signals, a popular choice is to take a forward finite difference.
Other popular choices of dictionary includes translation invariant wavelets (which can be viewed as a higher order total variation following~\cite{steidl2004equivalence}) or the concatenation of a derivative operator with the identity, known under the name of Fused Lasso~\cite{tibshirani2005sparsity} in statistics.

When there is no noise, i.e. $y = \Phi x_0$, it is common to use a constrained version of~\eqref{eq:p} which reads
\begin{equation}
  \label{eq:p0}
  \Xo = \uArgmin{x \in \RR^n} \normu{D^* x} \qsubjq \Phi x = y .
\end{equation}
It has been first introduced in~\cite{chen1999atomi} under the name Basis
Pursuit for $D = \Id$, and one can easily see that~\eqref{eq:p0} can be recasted as linear program (LP).

It is important to keep in mind that $\Xl$, nor $\Xo$ is typically not a singleton.
Most of previous works in this area is concerned with the case where the solution set is a singleton, or to derive guarantees to enforce uniqueness.
Necessary and sufficient conditions has been derived in~\cite{zhang2015necessary,zhang2013one} and also in~\cite{gilbert2015solution} for the constrained case.
In this paper, we tackle the case where $\Xl$ is not a singleton, and we want to better understand the structure of the solution set in this case.
Some insights are given in~\cite{tibshirani2013uniqueness}, but the results are limited to the case where $D=\Id$.
In this work, the authors give a bound on the size of the support, and prove that the LARS algorithm converges to a solution with a maximal support.
To our knowledge, our work is the first to consider the analysis case.
\section{Contributions}
\label{sec:contrib}

In~\cref{sec:sol}, we review some properties of the solution set.
In all this paper, \textbf{we consider the following hypothesis of restricted injectivity}
\begin{equation}\label{eq:hyp-inv}
  \Ker D^* \cap \Ker \Phi = \ens{0} ,
\end{equation}
in order to ensure that $\Xl$ is well-defined and bounded.
We prove in particular that $\Xl$ is a polytope, i.e. a bounded polyhedron.

Our main contribution is proved in~\cref{sec:max}.
It consist in providing a geometrical interpretation of a solution with a maximal $D$-support, namely the fact that such a solution lives in the relative interior of the solution set.
More precisely, we are concerned with the characterization of a vector of maximal $D$-support, i.e. a solution of~\eqref{eq:p} such that for every $x \in \Xl, \normz{D^* x} \leq \normz{D^* \maxx}$.
\begin{definition}
  A vector $\maxx \in \RR^n$ is \emph{a solution of maximal $D$-support} if $\maxx$ is a solution, i.e. $\maxx \in \Xl$ such that for every $x \in \Xl, \normz{D^* x} \leq \normz{D^* \maxx}$.
\end{definition}
We denote by $\Sl$ the set of solution of~\eqref{eq:p} which have maximal $D$-support.
Clearly this set is well-defined and contained in $\Xl$.
Our result is the following.
\begin{theorem}\label{thm:maximal-characterization}
  Let $\bar{x} \in \Xl$. Then $\bar{x}$ is a maximally $D$-supported solution if, and only if, $\bar{x} \in \rint \Xl$ (or equivalently if $\bar{x} \in \rint \Sl$).
  In other words, 
  \begin{equation*}
    \Sl = \rint \Sl = \rint \Xl .
  \end{equation*}
\end{theorem}
We recall that for any set $S$, the relative interior $\rint S$ of $S$ is
defined as its interior with respecto to the topology of the affine hull of $S$.

With this result in hand, we provide a way to construct such maximal solutions.
In \cref{sec:finding}, we show that with the help of a technical penalization using the so-called concave gauge ~\cite{barbara2015strict}, we can construct a path which converges to a point in the relative interior of $\Xl$, and more specifically, to the analytic center with respect to the chosen gauge.
We defer the precise statement to~\cref{sec:finding}.
\section{The Solution Set}
\label{sec:sol}

This section deals reviews some properties of the solution set $\Xl$.
The following proposition shows that even if $\Xl$ is not reduced to a singleton, its image by $\Phi$ or the analysis-$\lun$-norm is single-valued.
\begin{proposition}[Unique image]\label{lem:same-image}
  Let $x^1, x^2 \in \Xl$.
  Then,
  \begin{enumerate}
  \item they share the same image by $\Phi$, i.e., $\Phi x^1 = \Phi x^2$ ;
  \item they have the same analysis-$\lun$-norm, i.e., $\normu{D^* x^1} = \normu{D^* x^2}$.
  \end{enumerate}
\end{proposition}
A proof of this statement can be found for instance in~\cite{vaiter2011robust}.

It is known that standard $\ldeux$-regularization suffers from sign inconsistencies, i.e. two differents solutions can be of opposite signs at some indice.
The following proposition gives another important information: the cosign of two solutions cannot be opposite.
\begin{proposition}[Consistency of the sign]\label{prop:sign}
  Let $x^1, x^2 \in \Xl$.
  Then,
  \begin{equation*}
    \forall i \in \ens{1,\dots,p}, \quad u_i^1 u_i^2 \geq 0 ,
  \end{equation*}
  where $u^k = D^* x^k$ for $k=1,2$.
\end{proposition}
\begin{proof}
  The proof of this statement follows closely the proof found in~\cite{attouch1999p} for $\lun$.
  Suppose there exists $i$ such that $u_i^1$ and $u_i^2$ have opposite signs.
  Then, one has
  \begin{equation}\label{eq:sign-strict-ineq}
    \frac{\abs{u_i^1 + u_i^2}}{2} < \frac{\abs{u_i^1} + \abs{u_i^2}}{2} .
  \end{equation}
  Let $z = u^1 + u^2$.
  Using the convexity of $x \mapsto \norm{y - \Phi x}_2^2$ and inequality~\eqref{eq:sign-strict-ineq}, we get that
  \begin{align*}
    \frac{1}{2} \norm{y - \Phi z}_2^2 + \normu{D^* z} 
    & \!<\! \frac{1}{2} \left( \left( \frac{1}{2} \norm{y - \Phi x^1}_2^2 + \normu{D^* x^1}  \right) \!+\! \left( \frac{1}{2} \norm{y - \Phi x^2}_2^2 + \normu{D^* x^2} \right) \right) \\
    & \!=\! \umin{x \in \RR^n} \frac{1}{2} \norm{y - \Phi x}_2^2 + \normu{D^* x} ,
  \end{align*}
  which is a contradiction.
\end{proof}

Condition~\cref{eq:hyp-inv} (we recall that all through this paper, we suppose this condition holds) ensures that $\Xl$ is a non-empty, convex and compact set.
 Recall for all the following that given a lower semicontinuous real-valued extended convex function $h$ on $\RR^l$, its recession function can be defined by (Theorem 8.5 of \cite{rockafellar})
$$h_\infty(d)=\lim\limits_{\lambda\uparrow+\infty}\displaystyle{h(z+\lambda d)-h(z)\over\lambda},\ \forall (z,d)\in \dom(h)\times\RR^{l}.$$
In fact, as stated by the following proposition, the solution set $\Xl$ is a polytope.
\begin{proposition}\label{prop:polysol}
  $\Xl$ is a polytope (i.e. a bounded polyhedron).
\end{proposition}
\begin{proof}
  Let us first prove that $\Xl$ is a non-empty, convex and compact set. 
It follows with the help of hypothesis~\eqref{eq:hyp-inv} that $\{d:\ h_\infty(d)\leq0\}=\{0\}$.
  Hence, $\Xl$ is bounded.

  We shall now prove that $\Xl$ is a polytope.
  Let $\bar{x} \in \Xl$.
  According to Proposition~\ref{lem:same-image}, we have
  \begin{equation*}
    \Xl \subseteq
    \enscond{x \in \RR^n}{\normu{D^* x} = \normu{D^* \bar{x}}} \cap
    \enscond{x \in \RR^n}{\Phi x = \Phi \bar{x}} .
  \end{equation*}
  The reverse inclusion came from the fact that if $x$ shares the same image by $\Phi$ as $\bar{x}$ and the same analysis-$\lun$-norm, then the objective function at $x$ is equal to the one at $\bar{x}$, hence is also a solution.
  Thus,
  \begin{equation*}
    \Xl =
    \enscond{x \in \RR^n}{\normu{D^* x} = \normu{D^* \bar{x}}} \cap
    \enscond{x \in \RR^n}{\Phi x = \Phi \bar{x}} .
  \end{equation*}
  Hence, $\Xl$ is a polyhedron. Since $\Xl$ is a bounded set, it is also a polytope.
\end{proof}

Owing to Proposition~\ref{prop:polysol}, we can rewrite the set $\Xl$ as the convex hull of $k$ points in $\RR^n$ as
\begin{equation*}
  \Xl = \conv{a_1, \dots, a_k} ,
\end{equation*}
where $a_i$ are the extremal points of $\Xl$.
Observe that each $a_i$ lives on the boundary of the analysis-$\lun$-ball of radius $\normu{D^* \bar{x}}$.
Naturally, we can even rewrite the solution as
\begin{equation*}
  \Xl = A \Delta_k = \enscond{A z}{z \in \Delta_k} ,
\end{equation*}
where $A$ is a matrix $n \times k$ such that its columns are the vectors
$a_i$ and the $n$-simplex $\Delta_n$ of $\RR^n$ is defined as
\begin{equation*}
  \Delta_n = \enscond{x \in \RR^n}{\sum_{i=1}^n x_i = 1 \qandq \forall i, x_i \geq 0} = \conv{e_1,\dots,e_n} ,
\end{equation*}
where $(e_1,\dots,e_n)$ is the canonical basis of $\RR^n$.
Since $a_i$ are the extremal points of $\Xl$, notice that $A$ has maximal rank.
Observe in particular that the lines of the matrix $D^* A$ have same signs according to Proposition~\ref{prop:sign}.
\section{Maximal support and proof of \cref{thm:maximal-characterization}}
\label{sec:max}

We recall that a vector $\maxx \in \RR^n$ is a solution of maximal $D$-support if $\maxx$ is a solution, i.e., $\maxx \in \Xl$ such that for every $x \in \Xl, \normz{D^* x} \leq \normz{D^* \maxx}$.
The following proposition proves that \emph{the} $D$-maximal support is indeed
uniquely defined.
\begin{proposition}
  Let $x \in \Xl$.
  Then the two following propositions are equivalent.
  \begin{enumerate}
  \item $x$ is a solution of maximal $D$-support, i.e. $x \in \Sl$.
  \item For any $\bar{x} \in \Xl$, $\supp(D^* \bar{x}) \subseteq \supp(D^* x)$.
  \end{enumerate}
\end{proposition}
\begin{proof}
  The two directions are proved separately.\\
  $(1) \Rightarrow (2)$.
  Suppose there exists $i_0 \in \ens{1,\dots,p}$ such that $i_0 \in \supp(D^* \bar{x})$ and $i_0 \not\in \supp(D^* x)$.
  Observe that $\tilde x = \frac{1}{2}(\bar{x} + x)$ is also an element of $\Xl$ by convexity of $\Xl$.
  Using Proposition~\ref{prop:sign}, we get that $\supp(D^* \tilde x) \supseteq \supp(D^* \bar{x}) \cup \supp(D^* x)$.
  In particular, $\supp(D^* \tilde x) \supseteq \supp(D^* x) \cup \ens{i_0} \supsetneq \supp(D^* x)$.
  Hence, $\abs{\supp(D^* \tilde x)} > \abs{\supp(D^* x)}$ which contradicts the fact that $x$ has maximal $D$-support.\\
  $(2) \Rightarrow (1)$.
  Taking the cardinal in the property $\forall \bar{x} \in \Xl$, $\supp(D^* \bar{x}) \subseteq \supp(D^* x)$ is sufficient.
\end{proof}
In particular, two solutions of maximal support share the same $D$-support. Notice that in this case, the sign vectors are also the same.

We start by a technical Corollary of Proposition~\ref{prop:sign} which will be convenient in the following.
\begin{corollary}\label{cor:diagpos}
  There exists an integer $m \in \NN$, a matrix $\Lambda = \diag(\lambda_i)_{i=1,\dots,p}$ with $\lambda_i \in \ens{-1,1}$ for $i \in \ens{1,\dots,m}$ and $\lambda_i = 0$ for $i \in \ens{m+1,\dots,p}$, and a permutation matrix $\Sigma$ such that for $\Gamma = \Lambda \Sigma$, one has
  \begin{equation*}
    \Gamma D^* \Xl \subset (\RR_+)^m \times \ens{0}^{p-m} .
  \end{equation*}
  Moreover, for all $x \in \Xl$, $\normu{\Gamma D^* x} = \normu{D^* x}$.
\end{corollary}
\begin{proof}
  Let $\maxx$ an element of $\Sl$.
  Consider $I = \supp(D^* \maxx)$, $J = I^c$ and $m = \abs{I}$.
  Let $\Sigma$ be the permutation matrix associated to any permutation $\sigma$ which sends $I$ to $\ens{1,\dots,m}$.
  Define the matrix $\Lambda$ by its diagonal as
  \begin{equation*}
    \lambda_{\sigma(i)} = 
    \begin{cases}
      1 & \text{if } (D^* \maxx)_{\sigma(i)} > 0 \\
      -1 & \text{if } (D^* \maxx)_{\sigma(i)} < 0 \\
      0 & \text{if } (D^* \maxx)_{\sigma(i)} = 0 .
    \end{cases}
  \end{equation*}

  Now take any solution $x \in \Xl$ and consider the vector $u = \Gamma D^* x$.
  Let $i \in \ens{1,\dots,m}$, then
  \begin{equation*}
    u_i = \dotp{e_i}{\Lambda \Sigma D^* x} .
  \end{equation*}
  Since $\Lambda$ is self-adjoint, one has
  \begin{equation*}
    u_i = \dotp{\Lambda e_i}{\Sigma D^* x} .
  \end{equation*}
  Since $\Lambda$ is a diagonal matrix, we get that
  \begin{equation*}
    u_i = \lambda_i \dotp{e_i}{\Sigma D^* x} .
  \end{equation*}
  Now, since $\Sigma$ is a permutation matrix, we have that $\Sigma^* = \Sigma^{-1}$, i.e.
  \begin{equation*}
    u_i = \lambda_i \dotp{\Sigma^{-1} e_i}{D^* x} .
  \end{equation*}
  Using the permutation $\sigma$ associated to $\Sigma$, we have that
  \begin{equation*}
    u_i = \lambda_i \dotp{e_{\sigma^{-1}(i)}}{D^* x} ,
  \end{equation*}
  which can be rewritten as
  \begin{equation*}
    u_i = \lambda_i \dotp{d_{\sigma^{-1}(i)}}{x} .
  \end{equation*}  
  According to Proposition~\ref{prop:sign}, one have $(D^* x)_{\sigma^{-1}(i)} (D^* \maxx)_{\sigma^{-1}(i)} \geq 0$.
  Moreover, $\lambda_i = \lambda_{\sigma(\sigma^{-1}(i))}$ has the same sign than $(D^* \maxx)_{\sigma^{-1}(i)}$.
  Thus, $u_i = \lambda_i \dotp{d_{\sigma^{-1}(i)}}{x} \geq 0$.

  For $i \in \ens{m+1,\dots,p}$, we have that
  \begin{equation*}
    u_i = \lambda_i \dotp{e_i}{\Sigma D^* x} = 0,
  \end{equation*}
  since $\lambda_i = 0$.
\end{proof}
Note that the matrix $\Lambda$ and $\Sigma$ are not uniquely defined. Corollary~\ref{cor:diagpos} allows us to work only on positive vectors in dimension $m$.

We will also need to exclude at some point the case where a solution $x$ lives in the kernel of $D^*$.
The following lemma shows that if this is the case, then the solution set is reduced to a singleton $\Xl = \ens{x}$.
\begin{lemma}\label{lem:kernel-one-image}
  If there exists $x \in \Ker D^* \cap \Xl$, then $\Xl = \ens{x}$ .
\end{lemma}
\begin{proof}
  We recall that $\Xl \subset x + \Ker \Phi$.
  Let $\bar x \in \Xl$, and rewrite it as $\bar x = x + h$ where $h \in \Ker \Phi$.
  Then, according to Proposition~\ref{lem:same-image}, one has $\normu{D^* \bar x} = \normu{D^* x} = 0$.
  In particular, $\normu{D^* \bar x} = \normu{D^* x + D^* h} =  \normu{D^* h} = 0$.
  Using hypothesis~\eqref{eq:hyp-inv}, we get that $h = 0$.
\end{proof}

We can now provide the proof of Theorem~\ref{thm:maximal-characterization}.
\begin{proof}[Proof of Theorem~\ref{thm:maximal-characterization}]
  We exclude here the case where $\Xl$ is reduced to a singleton, since the result is then trivially verified.
  Let us prove both direction separately.

  $(\Leftarrow: \rint \Xl \subseteq \Sl)$.
  First, we recall that $\rint \Xl = \rint (A \Delta_k) = A \rint \Delta_k$.
  Let $\bar{x} \in \rint \Xl$.
  We have
  \begin{equation*}
    \bar{x} = A \bar{z} \qwithq \sum_{i=1}^k \bar{z}_i = 1 \qandq \bar{z}_i > 0.
  \end{equation*}
  For $i \in \ens{1,\dots,m}$, one has
  \begin{equation*}
    (\Gamma D^* \bar{x})_i = (\Gamma D^* A \bar{z})_i = \dotp{e_i}{\Gamma D^* A \bar{z}} =  \dotp{e_i}{\Lambda \Sigma D^* A \bar{z}}.
  \end{equation*}
  Using the fact that $\Lambda$ is a diagonal matrix and $\Sigma$ is a permutation matrix, we have that
  \begin{equation*}
    (\Gamma D^* \bar{x})_i = \lambda_i \dotp{D \Sigma^{-1} e_i}{A \bar{z}} ,
  \end{equation*}
  which can be rewritten, using the fact that $\Sigma^{-1} e_i = e_{\sigma^{-1}(i)}$ where $\sigma$ is the permutation associated to $\Sigma$, as
  \begin{equation*}
    (\Gamma D^* \bar{x})_i = \lambda_i \dotp{d_{\sigma^{-1}(i)}}{A \bar{z}} .
  \end{equation*}
  Now, one can rewrite it as
  \begin{equation*}
    (\Gamma D^* \bar{x})_i = \lambda_i \dotp{(D^* A)^* e_{\sigma^{-1}(i)}}{\bar{z}} .
  \end{equation*}
  Since for any $i$, $\bar{z}_i > 0$ and, according to Proposition~\ref{prop:sign}, there exists $j_0$ such that $((D^* A)^* e_{\sigma^{-1}(i)})_{j_0} > 0$, one concludes that $(\Gamma D^* \bar{x})_i \neq 0$.

  $(\Rightarrow: \Sl \subseteq \rint \Xl)$.
  We are going to prove that $\Sl = \rint \Sl$.
  Indeed, according to $(\Leftarrow)$, $\rint \Xl \subseteq \Sl$.
  Moreover, since every element of $\Sl$ is also an element of $\Xl$, we have $\rint \Xl \subseteq \Sl \subseteq \Xl$.
  In particular, $\aff \Xl = \aff \Sl$.
  Let
  \begin{equation*}
    \alpha = \min_{i \in \supp(D^* \maxx)} \abs{(D^* \maxx)_i} = \min_{i \in \ens{1,\dots,m}} (\Gamma D^* \maxx)_i
  \end{equation*}
  where $\maxx$ is an element of $\Sl$.
  Note that according to Lemma~\ref{lem:kernel-one-image}, since $\Xl$ is not reduced to a singleton, then $\supp(D^* \maxx)$ has cardinal greater than 1, hence $\alpha > 0$.

  Now take any $u \in B_\infty(\maxx, r) \cap \aff \Xl$ where
  \begin{equation*}
    r = \frac{\alpha - \epsilon}{\norm{\Gamma D^*}_{\infty,\infty}},
  \end{equation*}
  and $0 < \epsilon < \alpha$.

  Let's prove first that $\Gamma D^* u \in (\RR_+^*)^m \times \ens{0}^{p-m}$.
  From the definition of $u$, we get that
  \begin{equation*}
    \normi{\Gamma D^* u - \Gamma D^* x} \leq \norm{\Gamma D^*}_{\infty,\infty} \normi{u - x} \leq \alpha - \epsilon .
  \end{equation*}
  For $i \in \ens{1,\dots,m}$, one has $\abs{(\Gamma D^* u)_i - (\Gamma D^* x)_i} \leq \alpha - \epsilon$. In particular one has
  \begin{equation*}
    (\Gamma D^* u)_i - (\Gamma D^* x)_i \geq -\alpha + \epsilon \Leftrightarrow (\Gamma D^* u)_i \geq (\Gamma D^* x)_i - \alpha + \epsilon .
  \end{equation*}
  Since $(\Gamma D^* x)_i - \alpha \geq 0$ and $\epsilon > 0$, we conclude that $(\Gamma D^* u)_i > 0$.
  Thus, $(\Gamma D^* u)_i > 0$ for $i \in \ens{1,\dots,m}$ and $(\Gamma D^* u)_i = 0$ for $i \not\in \ens{1,\dots,m}$.

  It remains to prove that $u$ is a solution of~\eqref{eq:p}, i.e. $u \in \Xl$.
  Since $u \in \aff \Xl$, there exists $t \in \RR$ and $x \in \Xl$ such that
  \begin{equation*}
    u = \maxx + t (x - \maxx) .
  \end{equation*}
  From this equality, we get that
  \begin{align*}
    \normu{D^* u} &= \normu{\Gamma D^* u} = \sum_{i=1}^p (\Gamma D^* u)_i && \text{according to Corollary~\ref{cor:diagpos}}  \\
                  &= \sum_{i=1}^p (1-t) (\Gamma D^* \maxx)_i + t (\Gamma D^* x)_i && \\
                  &= (1-t) \normu{\Gamma D^* \maxx} + t \normu{\Gamma D^* x} && \\
                  &= \normu{D^* \maxx} && \text{since } \normu{D^* \maxx} = \normu{D^* x} .
  \end{align*}
  Moreover, $\Phi u = \Phi \maxx + t(\Phi x - \Phi \maxx) = \Phi \maxx$.
  Thus, $u$ is a solution which concludes our proof.
\end{proof}

\section{Finding a Maximal Solution}
\label{sec:finding}

Using the classical barrier function, in this section we show how to get a path that converges to a relative interior point of $\Xl$, which turns out to be the analytic center of $\Xl$.

Setting $Q = \Phi^* \Phi$ is the Gram matrix and $c = \Phi^* y$, we start by rewriting our initial problem~\cref{eq:p} as an augmented quadratic program under constraints, i.e.
\begin{equation*}
  \umin{x \in \RR^n, t \in \RR^p}
  \frac{1}{2} \dotp{Qx}{x} - \dotp{c}{x} + \lambda \sum_{i=1}^p t_i
  \qsubjq
  \begin{cases}
    -t \leq D^* x \leq t &\\
    t_i \geq 0 &\\
  \end{cases} ,
\end{equation*}
witch also can be rewritten as
\begin{equation*}
  \umin{x \in \RR^n, t \in \RR^p}
  \frac{1}{2} \dotp{Qx}{x} - \dotp{c}{x} + \lambda \sum_{i=1}^p t_i
  \qsubjq
  \begin{cases}
  -t+s=D^*x&\\
  t-s'=D^*x&\\
    t_i \geq 0,\ s_i\geq0,\ s'_i\geq0 &\\
  \end{cases}.
\end{equation*}
Now observe that $t=\displaystyle{1\over 2}(s+s')$. Then setting $z=\displaystyle{1\over 2}\left(\begin{array}{l}s\cr s'\end{array}\right)$, $I_p$ the $p$ by $p$ identity matrix, ${\tilde I}=\left(\begin{array}{lr}I_p& -I_p\end{array}\right)$ and $e=(1, \cdots,1)\in\RR^{2p}$, we come to the following equivalent formulation of the problem
\begin{equation}
\label{eq:paug}
\umin{x\in\RR^n, z\in\RR^{2p}}
f(x,z)
\qsubjq z\in[0,+\infty)^{2p}
\end{equation}
where $$
f(x,z)=\left\{\begin{array}{ll}\displaystyle{1\over2}\dotp{Qx}{x} - \dotp{c}{x} +\lambda\dotp{e}{z}&\mbox{ if }D^*x+{\tilde I}z=0\\
+\infty&\mbox{ elsewhere,}\end{array}\right.$$ 
or equivalently 
$$f(x,z)=\left\{\begin{array}{ll}\displaystyle{1\over2}\|\Phi x-y\|^2-\displaystyle{1\over2}\|y\|^2 +\lambda\dotp{e}{z}&\mbox{ if }D^*x+{\tilde I}z=0\\
+\infty&\mbox{ elsewhere.}\end{array}\right.$$

Its classical dual is
\begin{equation}
\label{eq:daug}
\umax{x\in\RR^n, s\in\RR^{2p}, u\in\RR^p}
g(x,s,u)
\qsubjq s\in[0,+\infty)^{2p}
\end{equation}
where 
$$g(x,s,u)=\left\{\begin{array}{ll}
-\displaystyle{1\over 2}\langle Qx,x\rangle&\mbox{if }{D} u+c-Qx=0,\ s=\lambda e-{\tilde I}^*u\cr 
-\infty&\mbox{elsewhere.}\end{array}\right.$$ 
We set $S_{(P)}$ (resp. $S_{(D)}$) the optimal solutions' set of
problem~\cref{eq:paug} (resp. problem~\cref{eq:daug}).
We know that $\Xl$ is non-empty and so $S_{(P)}$.
Since, in addition~\cref{eq:paug} is a convex problem with polyedral constraints, $S_{(D)}$ is non empty and there is no duality gap. We denote by $\alpha$ the optimal value of the two problems. 

\begin{proposition}\label{dcompacity}{$ $}

\begin{itemize}
\item[1.]The optimal solution $S_{(P)}$ of the problem (\ref{eq:paug}) is bounded or equivalently the set $\{(d_x,d_z):\ f_\infty(d_x,d_z)\leq0,\ d_z\geq0\}=\{0\}$,
\item[2.]$S(.,(D))=\{(s,u):\ \exists x\in\RR^n\mbox{ such that }(x,s,u)\in S_{(D)}\}$ is bounded, in other words, the dual feasible solutions' set is bounded in $(s,u)$.
\end{itemize}
\end{proposition}
\begin{proof}
1. Because of relation (\ref{eq:hyp-inv}) it is not difficult to show that the optimal solution $S_{(P)}$ of the problem (\ref{eq:paug}) is bounded.

2. Let $(x^k,s^k,u^k)$ be a sequence of the dual feasible solutions' set. We have $s^k=\lambda e-{\tilde I}^*u=\left(\begin{array}{l}\lambda e^p\cr\lambda e^p\end{array}\right)-\left(\begin{array}{l}u^k\cr-u^k\end{array}\right)\geq0$, where $e^p=(1,\cdots 1)\in\RR^p$. It follows that $-\lambda e^p\leq u^k\leq \lambda e^p$. Hence $(u^k)$ and then $(s^k)$, is bounded.
\end{proof}

Using the classical logarithmic barrier function introduced by Frish~\cite{frisch}, we deal with the family of problems $(P_\mu)_{\mu>0}$ given by
\begin{equation*}
\theta(\mu)=\umin{x\in\RR^n, z\in\RR^{2p}}
F_{\mu}(x,z)=f(x,z)+\zeta(z,\mu)
\end{equation*}
where $$\zeta(z,\mu)=\left\{\begin{array}{ll}
\mu \xi\left(z/\mu\right)&\mbox{if }\mu>0,\cr\xi_\infty(z)&\mbox{if }\mu=0,\cr+\infty&\mbox{elsewhere,}
\end{array}\right.$$ $$ \xi(z)=\left\{\begin{array}{ll}-\ln \varphi(z)&\mbox{if }\varphi(z)>0,\cr+\infty&\mbox{elsewhere,}\end{array}\right. \mbox{ and }\varphi(z)=\left\{\begin{array}{ll}\left(\prod\limits_{i=1}^{2p}z_i\right)^{1\over2p}&\mbox{if }z\geq0,\cr-\infty&\mbox{elsewhere.}\end{array}\right.$$

Note that the function $\varphi$ is strictly quasiconcave and then according to Lemma 1 of \cite{barbara2015strict}, for every $\mu>0$, the function $\zeta_\mu:z\mapsto\zeta(z,\mu)$ is strictly convex on $(0,+\infty)^{2p}$. 
\begin{proposition}\label{strict_convexity}
For every $\mu>0$, the function $F_{\mu}$ is inf-compact on $\RR^n\times\RR^{2p}$ and strictly convex on $\RR^n\times(0,+\infty)^{2p}$.
\end{proposition}
\begin{proof} Let us show that  
\begin{eqnarray}\label{xiinfty}
\xi_\infty(d)=\left\{\begin{array}{ll}0&\mbox{if }d\geq0,\cr+\infty&\mbox{elsewhere.}\end{array}
\right.\end{eqnarray}
Let $(z,d)\in \dom(\xi)\times\RR^{2p}$. We have necessarily $z>0$. First we observe that when $d\not\in[0,+\infty)^{2p}$, $z+\lambda d\not\in[0,+\infty)^{2p}$ for $\lambda$ large enough and then $\xi_\infty(d)=+\infty$. Now consider the case $d\geq0$. Since $z>0$ we have necessarily $z+d>0$. The concave gauge function $\varphi$ is monotone with respect to its domaine the positive orthant. Then by Proposition 2.1 of \cite{barbara_crouzeix},
$$0<\varphi(z+d)\leq\varphi(z+\lambda d)\leq\varphi(\lambda z+\lambda d)=\lambda\varphi(z+d)$$
for $\lambda$ large enough. It follows that
$$\begin{array}{ll}
0=\lim\limits_{\lambda\uparrow+\infty}\displaystyle{\ln\varphi(z+d)-\ln\varphi(z)\over\lambda}&\leq\lim\limits_{\lambda\uparrow+ \infty}\displaystyle{\ln\varphi(z+\lambda d)-\ln\varphi(z)\over\lambda}\cr&\leq\lim\limits_{\lambda\uparrow+\infty} \displaystyle{\ln\lambda\varphi(z+d)-\ln\varphi(z)\over\lambda}=0\end{array}$$ and hence $\lim\limits_{\lambda\uparrow+ \infty}\displaystyle{\ln\varphi(z+\lambda d)-\ln\varphi(z)\over\lambda}=0$. Consequently $\xi_\infty(d)=0$.

By Proposition \ref{dcompacity}, we have $\{(d_x,d_z):\ f_\infty(d_x,d_z)\leq0,\ d_z\geq0\}=\{(0,0)\}$. Thus
$\{(d_x,d_z):\ {F_\mu}_\infty(d_x,d_z)\leq0,\ d_z\geq0\}=\{(0,0)\}$, or equivalently, $F_\mu$ is inf-compact. 

Now let us proceed to prove the strict convexity of $F_\mu$.  Take  $(x,z)\not=(x',z')$ in $\RR^n\times(0,+\infty)^{2p}$ and $t\in(0,1)$. In the case where $z\not= z'$, by strict-convexity of $\zeta_\mu$ on $(0,+\infty)^{2p}$ we have necessarily $F_\mu(t(x,z)+(1-t)(x',z'))<tF_\mu(x,z)+(1-t)F_\mu(x',z').$ Assume that $z=z'$. Using (\ref{eq:hyp-inv}) and the definition of $f$ we obtain $\Phi x\not=\Phi x'$ and the result follows by using the strict convexity of $\|.\|_2^2$.
\end{proof}
Propositions \ref{strict_convexity} and \ref{dcompacity} assert that for every $\mu>0$ there is a unique optimal solution  $(x(\mu),z(\mu))$ to $(P_\mu)$. Moreover using the fact that $F_\mu(x,\cdot)$ is a barrier function for every $x\in\RR^n$, $z(\mu)>0$. Consider the function $\gamma:\RR^n\times[0,+\infty)^{2p}\times[0,+\infty)\to \RR\cup\{+\infty\}$ defined by
$$\gamma(x,z,\mu)=F_{\mu}(x,z).$$
Then we have the following proposition.

\begin{proposition}
\label{coercivity}
The function $\gamma$ is convex and lsc on $\RR^n\times\RR^{2p}\times[0,+\infty)$. It is inf-compact on $\RR^n\times\RR^{2p}\times[0,{\overline \mu}]$, $\forall\overline{\mu}>0$ being fixed. Moreover $\theta$ is convex and continuous on $[0,+\infty)$, $\theta(0)=\alpha$ and $f(x,z)=\gamma(x,z,0)$, $\forall (x,z)\in\RR^n\times(0,+\infty)^{2p}$.
\end{proposition}
\begin{proof}
It is known that the function $\zeta$ is convex on $\RR^{2p}\times[0,+\infty)$
and so is $\gamma$. The function $\theta$ is then convex on $[0,+\infty)$ as the
infimum over $(x,z)$ of a convex function in $(x,z,\mu)$. Now the function
$\zeta(z,.)$ is continuous on $[0,+\infty)$ and, because of (\ref{xiinfty}),
$\zeta(z,0)=0$ for all $z\in(0,+\infty)^{2p}$. Thus $f(x,z)=\gamma(x,z,0)$ for
all $(x,z)\in\RR^n\times(0,+\infty)^{2p}$ and therefore $\theta(0)=\alpha$ (the
optimal value of the problem (\ref{eq:paug})). Set
$\tilde{\gamma}=\gamma_{|\RR^n\times\RR^{2p}\times[0,\overline{\mu}]}$ the
restriction of $\gamma$ to the set
$\RR^n\times\RR^{2p}\times[0,\overline{\mu}]$. Then $\{(d_x,d_z,\mu):\
\tilde{\gamma}_\infty(d_x,d_z,\mu)\leq0,\ d_z\geq0,\ \mu=0\}=\{(d_x,d_z,0):\
f_\infty(d_x,d_z)\leq0,\ d_z\geq0\}=\{(0,0,0)\}$ (see Proposition
\ref{dcompacity}). The function $\gamma$ is then inf-compact on
$\RR^n\times\RR^{2p}\times[0,{\overline \mu}]$. Consequently, there is a compact
$\tilde{S}$ such that $(x(\mu),z(\mu))\in\tilde{S}$,
$\forall\mu\in(0,\overline{\mu}]$, i.e.,
$(x(\mu),z(\mu))_{\mu\in(0,\overline{\mu})}$ is bounded. We established that $\theta$ is convex on $[0,+\infty)$. It is then continuous on $(0,+\infty)$. Let us show now that $\lim\limits_{\mu\downarrow 0}\theta(\mu)=\theta(0)=\alpha$. In this respect we shall prove that 
$\lim\limits_{\mu\downarrow0}\mu\ln\left(\displaystyle{\varphi(z(\mu)) \over\mu}\right)= 0$. Let $(\mu^k)_{k\in\NN}$ be a positive sequence such that $\lim\limits_{k\uparrow+\infty}\mu^k=0.$ We established that $(x(\mu),z(\mu))_{\mu\in(0,\overline{\mu}]}$ is bounded. It follows that the set $\{(x(\mu^k),z(\mu^k))\}$ contains a subsequence converging to a point  $(\tilde{x},\tilde{z})$.
In the case where $\tilde{z}>0$ the result is obvious. Assume that $\varphi(\tilde{z})=0$. Then for $k$ sufficiently large one has
$$\begin{array}{ll}\alpha-\mu^k\ln\left(\displaystyle{\varphi(z)\over\mu^k}\right)
\leq \theta(\mu^k)&=f(x(\mu^k),z(\mu^k))-\mu^k\ln\left(\displaystyle{\varphi(z(\mu^k)) \over \mu^k}\right)\cr&
\leq f(x,z)-\mu^k\ln\left(\displaystyle{\varphi(z)\over\mu^k}\right)
\end{array}$$
for every $(x,z)$ satisfying $z>0$. Since $\lim\limits_{k\uparrow 0}\mu^k\ln\left(\displaystyle{\varphi(z)\over\mu^k}\right)=0$, we have
$$\alpha\leq\lim\inf\limits_{k\uparrow+\infty}\theta(\mu^k)\leq f(x,z)$$
and then
$$\alpha\leq\lim\sup\limits_{k\uparrow+\infty}\theta(\mu^k)\leq \inf\limits_{x,z}\{f(x,z):\ z>0\}=\inf\limits_{x,z}\{f(x,z):\ z\geq0\}=\alpha.$$
Consequently $\lim\limits_{k\uparrow+\infty}\theta(\mu^k)=\alpha$.
\end{proof}
\bigskip
\bigskip

Given $\mu>0$, the KKT optimalty conditions for the problem $(P_\mu)$ can be formulated, for some $u\in\RR^p$, as
$$\left\{\begin{array}{ll}Qx(\mu)-c-Du=0,\\ \lambda e-\displaystyle{ \mu\over 2p}(Z(\mu))^{-1}e-{\tilde I}^*u=0,\\ D^*x(\mu)+{\tilde I}z(\mu)=0,\end{array}\right.$$
where $Z(\mu)=diag(z(\mu))$.
Observe that $u$ is necessarily unique. Put 
$$u=u(\mu)\mbox{ and }s(\mu)=\displaystyle {\mu\over 2p}Z^{-1}(\mu)e.$$ We rewrite the KKT conditions as
$$\left\{\begin{array}{lr}Qx(\mu)-c-Du(\mu)=0&(E1)\\ \lambda e-s(\mu)-{\tilde I}^*u(\mu)=0&(E2)\\ Z(\mu)s(\mu)=\displaystyle {\mu\over 2p}e&(E3)\\D^*x(\mu)+{\tilde I}z(\mu)=0&(E4)\end{array}\right.$$

\begin{proposition}
For every $\mu>0$, $(s(\mu),u(\mu))$ is a feasible solution to (\ref{eq:daug}) and $\big((s(\mu),u(\mu)\big)_{\mu\in(0,\overline{\mu}]}$ is bounded.
\end{proposition}

\begin{proof}
By $(E1)$, $(E2)$ and the fact that $s(\mu)=\displaystyle {\mu\over 2p}(Z(\mu))^{-1}e>0$, $(u(\mu),s(\mu))$ is a feasible solution to (\ref{eq:daug}). The boundedness of $(s(\mu),u(\mu))_{\mu\in(0,{\overline\mu}]}$ is due to Proposition \ref{dcompacity}.
\end{proof}

Set ${\overline I}=\displaystyle\bigcup_{\atop z\in S(.,(P))}I(z)$ and ${\overline J}=\displaystyle\bigcup_{\atop s\in S(.,(D))}J(s)$, where 
$$S(.,(P))=\left\{z:\ \exists x\in\RR^n\mbox{ such that } (x,z)\in S_{(P)}\right\},$$ 
$$S(.,(D))=\left\{s:\ \exists u\in\RR^p\mbox{ such that } (s,u)\in S_{(D)}\right\},$$ $$I(z)=\{i:\ z_i>0\}\mbox{ the support of }z\mbox{ and }J(s)=\{i:\ s_i>0\}\mbox{ the support of }s.$$  

\begin{lemma}\label{complementarity}

There is at least one $({\hat z},\hat{s})\in S(.,(P))\times S(.,(D))$ such that ${\overline I}=I({\hat z})$ and $\overline{J}=J(\hat{s})$.
\end{lemma}

\begin{proof}
We have ${\overline I}$ a subset of a finite set $\{1,\cdots,2p\}$. Let then $(z^1,\ z^2,\cdots,z^k)\in S(.,(P))^k$, for some $k\in\{1,2,\cdots,2p\}$
satisfying ${\overline I}=I\left(z^1\right)\cup I\left(z^2\right)\cup\cdots\cup I\left(z^k\right)$. Set ${\hat z}=\displaystyle{1\over k}\left(z^1+z^2+\cdots+z^k\right)$. Since $S(.,(P))$ is convex  ${\hat z}\in S(.,(P))$. So it is easy to see that $I(z^i)\subset I({\hat z})$, $\forall i\in\{1,2,\cdots,k\}$. The result then follows. A vector $\hat{s}$ is constructed in a similar way.
\end{proof} 
Observe that every optimal solution $(x,z)$ of the problem (\ref{eq:paug}) satisfying $I(z)=\overline{I}$ is in the relative interior of $S_{(P)}$. Similarily every optimal solution $(x,s,u)$ of the problem (\ref{eq:daug}) satisfying $J(s)=\overline{J}$ is in the relative interior of $S_{(D)}$.

\bigskip

Set 
$$({\overline x},{\overline z})=\arg\max\left\{\varphi_{\overline I}(z_{\overline I}):\ \displaystyle{1\over 2}\langle Qx,x\rangle-\langle c,x\rangle+\lambda\langle e,z\rangle=\alpha,\ D^*x+{\tilde I}z=0,\ z_{\overline{J}}=0\right\},$$
where 
$$\varphi_{\overline I}(z_{\overline I})=\left\{\begin{array}{ll}
\left(\prod\limits_{i\in {\overline I}}z_i\right)^{1\over \card({\overline I})}&\mbox{if }z_J\in(0,+\infty)^{\card(J)}\cr
-\infty&\mbox{elsewhere.}
\end{array}\right.
$$
Symmetrically we set
$$({\overline s},{\overline u})=\arg\max\left\{\varphi_{\overline J}(s_{\overline J}):\ 
s=\lambda e-\tilde{I}^*u,\ D u+c-Q{\overline x}=0, s_{\overline I}=0\right\},$$
where
$$\varphi_{\overline{J}}(s_{\overline{J}})=\left\{\begin{array}{ll}\left(\prod\limits_{i\in \overline{J}}s_i\right)^{1\over \card(\overline{J})}&\mbox{if }s_{\overline{J}}\in(0,+\infty)^{\card(\overline{J})}\cr-\infty&\mbox{elsewhere.}\end{array}\right.
$$
$(\overline{x},\overline{z})$ is called the analytic center\footnotemark[1]\footnotetext[1]{A generalization of the central path and the analytic center is proposed in \cite{barbara2015strict} by using the so called concave gauge functions.} of (\ref{eq:paug}) and $(\overline{x},\overline{s},\overline{u})$ the analytic center of (\ref{eq:daug}). The uniqueness is ensured by the strict quasiconcavity of functions $\varphi_{\overline{I}}$ and  $\varphi_{\overline{J}}$ on the interior of their respective domain and the assumption (\ref{eq:hyp-inv}). We now give an important result.

Its proof is inspired in part by those of Theorems I.7 and I.9 in \cite{RoTevi}.
\begin{theorem}\label{thm:convergence}
Under assumption \ref{eq:hyp-inv}, we have $$\lim\limits_{\mu\downarrow0}(x(\mu),z(\mu),s(\mu),u(\mu))=({\overline x},{\overline z},{\overline s},{\overline u}).$$ Moreover, $({\overline x},{\overline z})$  and $({\overline x},{\overline s},{\overline u})$ belong to the relative interior of $S_{(P)}$ and $S_{(D)}$, respectively. 
\end{theorem}

\begin{proof}
We proved that $\big((x(\mu),z(\mu)\big)_{\mu\in(0,\overline{\mu}]}$ and $\big((s(\mu),u(\mu)\big)_{\mu\in(0,\overline{\mu}]}$ are bounded. Let $(\mu^k)_{k\in \NN}$ a positive increasing sequence satisfying $$\lim\limits_{k\uparrow +\infty}\mu^k=0\mbox{ and }\lim\limits_{k\uparrow+\infty}(x(\mu^k),z(\mu^k),s(\mu^k),u(\mu^k))=(\tilde{x},\tilde{z},\tilde{s},\tilde{u}).$$ Then replacing $\mu$ by $\mu^k$ in $(E1)-(E4)$ and letting $k$ tend to $+\infty$, we observe that the pair $\{(\tilde{x},\tilde{z}),(\tilde{x},\tilde{s},\tilde{u})\}$ satisfies the KKT optimality conditions of (\ref{eq:paug}) and then it is a primal-dual optimal solution pair of (\ref{eq:paug}). Let us show now that $I(\tilde{z})=\overline{I}$ and $J(\tilde{s})=\overline{J}$. 
Now by $(E1)$, $(E2)$ and $(E4)$ we have  
$$\left(\begin{array}{l}
x(\mu^k)-\overline{x}\cr
z(\mu^k)-\overline{z}\end{array}\right)\in\Ker{\left(\begin{array}{ll}D^*&{\tilde I}\end{array}\right)}\mbox{ and }
\left(\begin{array}{l}
Q(x(\mu^k)-\overline{x})\cr
-(s(\mu^k)-\overline{s})\end{array}\right)\in\Im{\left(\begin{array}{l}D \\ {\tilde I}^*\end{array}\right)}.
$$
Then using the following orthogonality property 
\begin{equation}
\label{orthogonality}
\Ker{\left(\begin{array}{ll}{D^*}&{\tilde I}\end{array}\right)}=\left[\Im{\left(\begin{array}{l}D \\ {\tilde I}^*\end{array}\right)}\right]^\bot,
\end{equation} 
$(E3)$ and the fact that $\langle\overline{z},\overline{s}\rangle=\langle\tilde{z},\tilde{s}\rangle=0$ we have 
$$\langle\overline{z},s(\mu^k)\rangle+\langle\overline{s},z(\mu^k)\rangle=\mu^k-\langle Q(x(\mu^k)-\overline{x}),x(\mu^k)-\overline{x}\rangle.$$ 
Since in addition $I(\overline{z})=\overline{I}$, $J(\overline{s})=\overline{J}$ and $Q$ is positive semi-definite we get
$$\sum\limits_{i\in \overline{I}}\overline{z}_is(\mu^k)_i+\sum\limits_{i\in \overline{J}}\overline{s}_iz(\mu^k)_i=\mu^k-\langle Q(x(\mu^k)-\tilde{x}),x(\mu^k)-\tilde{x}\rangle\leq\mu^k.$$
But from $(E3)$, $z(\mu^k)_is(\mu^k)_i=\displaystyle{\mu^k\over 2p},\ \forall i$. it follows that
$$\displaystyle\sum\limits_{i\in \overline{J}}{\overline{s}_i\over s(\mu^k)_i}+\sum\limits_{i\in \overline{I}}{\overline{z}_i\over z(\mu^k)_i}\leq 2p.$$
Now letting $k$ tend to $+\infty$, we get on the one hand
$$0<\displaystyle\sum\limits_{i\in \overline{J}}{\overline{s}_i\over \tilde{s}_i}+\sum\limits_{i\in \overline{I}}{\overline{z}_i\over \tilde{z}_i}\leq 2p<+\infty$$
and then, by construction of $\overline{I}$ and $\overline{J}$, we have necessarily $I(\tilde{z})=\overline{I}$ and $J(\tilde{s})=\overline{J}$. On the other hand,
using the arithmetic-geometric mean inequality we get
$$\left(\prod\limits_{i\in \overline{J}}\displaystyle{\overline{s}\over \tilde{s}_i}\prod\limits_{i\in \overline{I}}\displaystyle{\overline{z}\over \tilde{z}_i}\right)^{1\over 2p}\leq{1\over 2p}\left(\displaystyle\sum\limits_{i\in \overline{J}}{\overline{s}\over \tilde{s}_i}+\sum\limits_{i\in \overline{I}}{\overline{z}\over \tilde{z}_i}\right)\leq 1$$
and then 
$$\varphi_{\overline{J}}( \overline{s}_{\overline{J}})\varphi_{\overline{I}}( \overline{z}_{\overline{I}})\leq\varphi_{\overline{J}}( \tilde{s}_{\overline{J}})\varphi_{\overline{I}}( \tilde{z}_{\overline{I}}).$$
But, by definition of $(\overline{x},\overline{z},\overline{s},\overline{u})$, $\varphi_{\overline{J}}( \tilde{s}_{\overline{J}})\leq \varphi_{\overline{J}}( \overline{s}_{\overline{J}})$ and $\varphi_{\overline{I}}( \tilde{z}_{\overline{I}})\leq \varphi_{\overline{I}}( \overline{z}_{\overline{I}})$.  The result then follows. 
\end{proof}

Consequently, the following corollary holds
\begin{corollary}
  Under assumption (\ref{eq:hyp-inv}), we have $\lim\limits_{\mu \downarrow 0} x(\mu) = \bar x \in \rint \Xl$.
\end{corollary}

\begin{proof}
By  Theorem \ref{thm:convergence} $(\overline{x},\overline{z})$ belongs to  the relative interior of $S_{(P)}$ and hence $\overline{x}$ belongs to the linear projection of the relative interior of $S_{(P)}$ which is equal to $\rint \Xl$.
\end{proof}
Using the analysis, we propose an algorithm directly adapted from the Predictor-corrector Mehrotra's algorithm~\cite{Mehrotra}.
The pseudo-code is given in Algorithm~\ref{alg:main}.
The user is expected to give a primal-dual starting point $(x^0,z^0,u^0,s^0)$ satisfying $z^0 > 0$ and $s^0 > 0$, the scenario $\Phi$, $D^*$, $y$, a stopping criterion $\epsilon > 0$, and a relaxation parameter $\eta \in (0,1)$.

To illustrate our theoretical results, we consider a very simple scenario in $\RR^2$ to $\RR$.
Let $D = \Id_2$, $\Phi = (1 \quad 1)$, $y = 1$ and $\lambda = \frac{1}{2}$.
The first order conditions reads
\begin{align*}
  2 x_1 + 2 x_2 - 2 + s_1 &= 0 \\
  2 x_1 + 2 x_2 - 2 + s_2 &= 0 ,
\end{align*}
where $s \in \partial \normu{\cdot}(x)$.
One can check that $x^\star = (\frac{1}{2} \quad 0)^*$ is a solution.
Using the fact that $\Xl \subseteq x^\star + \Ker \Phi$ and that every solution share the same $\lun$-norm, we have that $\Xl = \conv{(\frac{1}{2} \quad 0)^*, (0 \quad \frac{1}{2})^*}$.
Figure~\ref{fig:path} represents the evolution of the primal iterate on the plane $\RR^2$.
\begin{figure}[h]
  \centering
  \begin{subfigure}[t]{0.45\textwidth}
    \centering
    \includegraphics[height=2in]{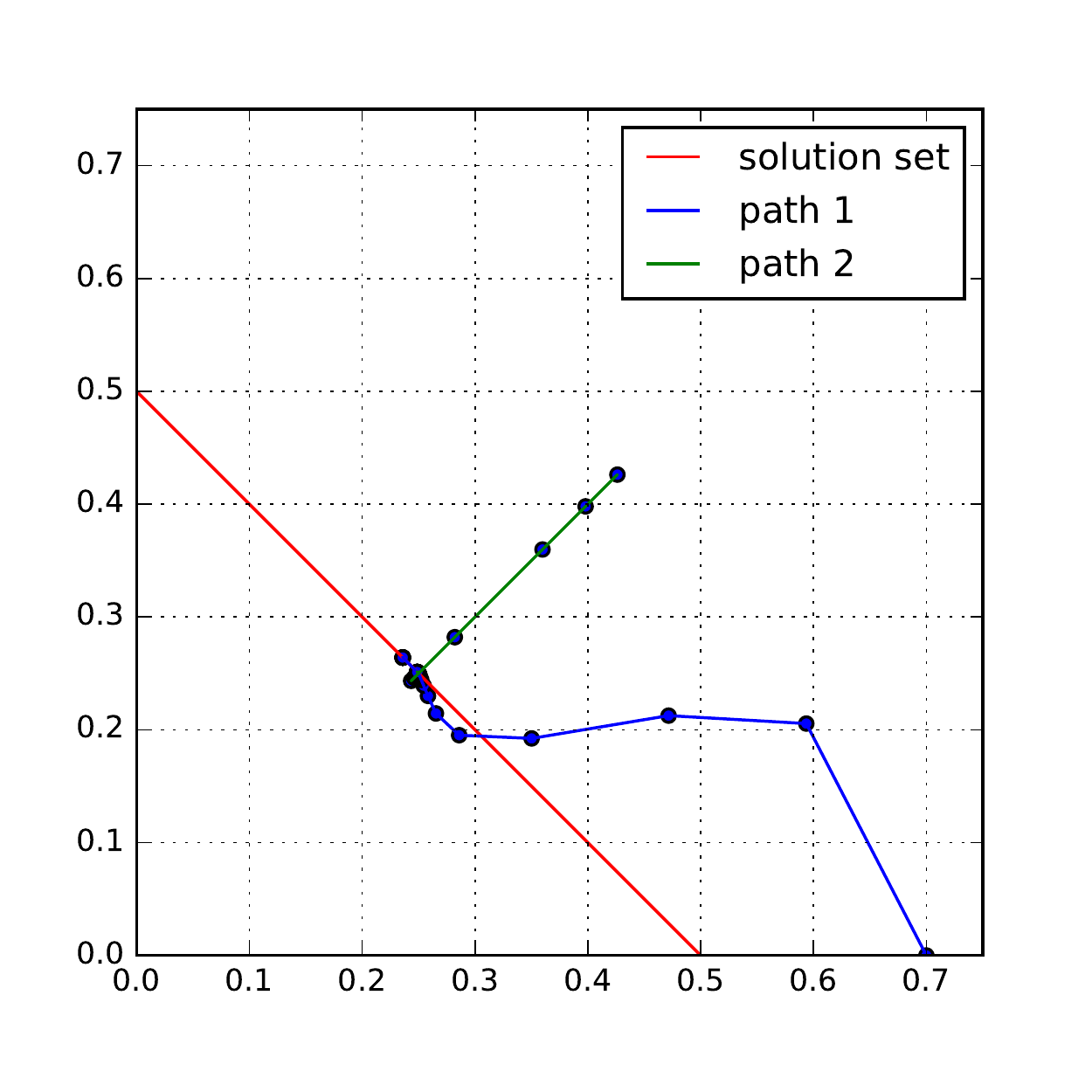}
    \caption{Path}
  \end{subfigure}%
  ~ 
  \begin{subfigure}[t]{0.45\textwidth}
    \centering
    \includegraphics[height=2in]{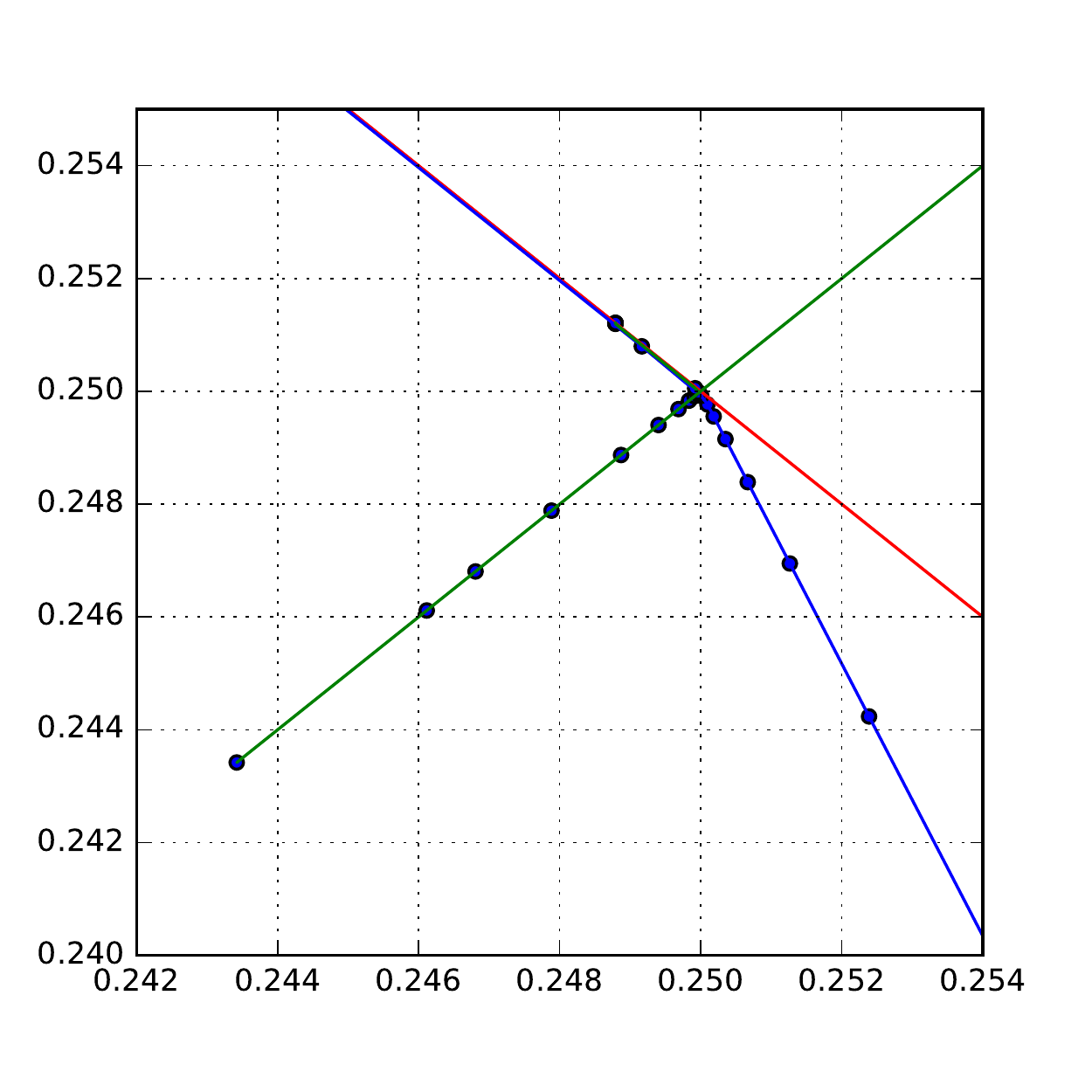}
    \caption{Zoom around the analytical center}
  \end{subfigure}
  \caption{Algorithm path. The red line corresponds to the solution set $\Xl$, the blue line is the algorithm path for $x^0 = (0.7 \, 0)^*$ and the green line for $x^0$ obtained by a least square.}
  \label{fig:path}
\end{figure}

\renewcommand{\algorithmicrequire}{\textbf{Input:}}

\begin{algorithm}
  \caption{Adapted predictor-corrector Mehrotra's algorithm}\label{alg:main}
  \begin{algorithmic}
    \Require $(x^0,z^0,u^0,s^0)$, $\Phi$, $D^*$, $y$, $\epsilon > 0$, $\eta \in (0,1)$
    \State $Q \gets \Phi^* \Phi$, $c \gets \Phi^* y$
    \State Set complementarity measure
    \[
          r_1 \gets Qx-c-D^*u,
          r_2 \gets \lambda e-s-\tilde{I}^*u,
          r_3 \gets Zs,
          r_4 \gets D^*x+\tilde{I}z.
    \] 
    \State $\mu \gets \displaystyle{\langle z,s\rangle\over 2p}$

    \While{$\max\{\|r_1\|_2,\|r_2\|_2,\|r_3\|_2,\|r_4\|_2\}>\epsilon$}
    \State Compute the affine scaling direction $(d_x^a,d_z^a,d_u^a,d_s^a)$ by solving the system
    \[
      \left\{
        \begin{array}{ll}
          Qd_x^a-D^*d_u^a&=-r_1 \\
          -d_s^a-\tilde{I}^*d_u^a&=-r_2 \\
          Sd_z^a+Zd_s^a&=-r_3 \\
          D^*d_x^a+\tilde{I}d_z^a&=-r_4,
        \end{array}
      \right.
    \]
    \State $t^a_{\max} \gets \max\{ t\geq 0:\ z+td_z^a\geq 0,\ s+d_s^a\geq 0\}$
    \State $\mu^a \gets \displaystyle{\langle z+t_{\max}^ad_z^a,s+t_{\max}^ad_s\rangle\over 2p}$
    \State $\sigma \gets \displaystyle{\left(\mu^a\over\mu\right)^3}$ \Comment{centering parameter}
    \State Compute corrector and centering direction $(d_x^c,d_z^c,d_u^c,d_s^c)$ by solving
    \[
      \left\{
        \begin{array}{ll}
          Qd_x^c-{D^*}^*d_u^c&=0\\
          -d_s^c-\tilde{I}^*d_u^c&=0\\
          Sd_z^c+Zd_s^c&=-D_z^ad_s^a+\sigma\mu e\\
          D^*d_x^a+\tilde{I}d_z^a&=0,
        \end{array}
      \right. \qwhereq D_z^a=\diag(d_z^a)
    \]
    \State $(d_x,d_z,d_u,d_s) \gets (d_x^a,d_z^a,d_u^a,d_s^a)+(d_x^c,d_z^c,d_u^c,d_s^c)$ \Comment{predictor direction}
    \State $t_{\max} \gets \max\{ t\geq 0:\ z+td_z\geq 0,\ s+d_s\geq 0\}$
    \State $(x,z,u,s) \gets (x,z,u,s)+\eta t_{\max}(d_x,d_z,d_u,d_s)$
    \State Update complementarity measure
    \[
          r_1 \gets Qx-c-D^*u,
          r_2 \gets \lambda e-s-\tilde{I}^*u,
          r_3 \gets Zs,
          r_4 \gets D^*x+\tilde{I}z.
    \]
    \EndWhile
  \end{algorithmic}
\end{algorithm}

\bibliographystyle{siamplain}
\bibliography{biblio}

\end{document}